\documentclass{amsart}

\usepackage{amsmath, amssymb, amsthm, mathtools, tikz-cd}
\usepackage{environ}
\usepackage{enumitem}
\usepackage[hidelinks]{hyperref}

\theoremstyle{plain}
\newtheorem{thm}{Theorem}
\newtheorem{prop}[thm]{Proposition}
\newtheorem{conj}[thm]{Conjecture}
\newtheorem{cor}[thm]{Corollary}
\newtheorem{lem}[thm]{Lemma}

\theoremstyle{definition}
\newtheorem{defi}[thm]{Definition}
\newtheorem{rmk}[thm]{Remark}
\newtheorem*{acknowledgments}{Acknowledgments}

\newcommand{\cH}{\mathcal{H}}
\newcommand{\cF}{\mathcal{F}}
\newcommand{\cK}{\mathcal{K}}
\newcommand{\cO}{\mathcal{O}}
\newcommand{\cD}{\mathcal{D}}
\newcommand{\cZ}{\mathcal{Z}}
\newcommand{\bC}{\mathbb{C}}
\newcommand{\bT}{\mathbb{T}}
\newcommand{\bN}{\mathbb{N}}

\newcommand{\clo}[1]{\overline{#1}}
\newcommand{\Span}{\operatorname{span}}

\newcommand{\into}{\hookrightarrow}

\newcommand{\abs}[1]{\left| {#1}\right|}

\NewEnviron{equa}{ \begin{equation*}\begin{split} \BODY \end{split} \end{equation*} }

\title{$\cZ$-stable Graph Algebras}
\author{Gregory Faurot}
\address{Department of Mathematics, The Ohio State University, Columbus, OH 43210}
\email{faurot.3@osu.edu}

\date{\today}


\begin{document}
	
	\begin{abstract}
		We introduce a divisibility-type condition for directed graphs that is necessary for $\cZ$-stability of the corresponding graph $C^*$-algebra. We prove that this condition is sufficient if either the graph $E$ has no cycles or the algebra $C^*(E)$ has finitely many ideals. Under the further assumption that $E$ is a finite graph, we provide a complete characterization of $\cZ$-stability of $C^*(E)$. We conjecture that our divisibility condition and Condition (K) are equivalent to $\cZ$-stability of the graph algebra. We prove that it is equivalent to $C^*(E)$ being pure, verifying the Generalized Toms--Winter Conjecture for graph algebras with finitely many ideals.
	\end{abstract}
	
	\maketitle
	
	\renewcommand{\thethm}{\Alph{thm}}
	\section*{Introduction}
	As part of the Toms--Winter Conjecture, it is known that tensorial absorption of the Jiang--Su algebra $\cZ$ is equivalent to finite nuclear dimension in the case of separable, unital, simple, nuclear, infinite-dimensional $C^*$-algebras (\cite{CETWW21, CE20, T14, MS14}). Outside the simple setting, elementary subquotients -- quotients of ideals that are isomorphic to the algebra of compact operators $\cK(\cH)$ on some Hilbert space $\cH$ -- provide a clear obstruction to $\cZ$-stability (\cite{TW07}), but not to finite nuclear dimension (\cite{WZ10}). 
    It is conjectured (\cite[Question~D]{APTV24}) that the additional hypothesis of no elementary subquotients is sufficient for the equivalence of finite nuclear dimension and $\cZ$-stability to hold for non-simple $C^*$-algebras. The first results in this direction are due to Robert and Tikuisis in \cite{RT17}, which provided sufficient conditions for $\cZ$-stability in a variety of contexts. Additional results surrounding this generalized conjecture appear in \cite{BL24, BGSW22, ENST20, APTV24}.
    \par Graph $C^*$-algebras form a nice class of examples to analyze. Introduced by Kumjian, Pask, and Raeburn (\cite{KPR98}), and further generalized by Drinen and Tomforde (\cite{DT05}), any directed graph $E$ has an associated $C^*$-algebra $C^*(E)$. A primary goal in the study of graph algebras is to relate $C^*$-algebraic properties of $C^*(E)$ to combinatorial properties of the underlying graph $E$. Furthermore, due to these connections, graph algebras are a natural test case for regularity conditions like $\cZ$-stability and finite nuclear dimension. Much progress has been made in computing the nuclear dimension of graph algebras; see \cite{RST15, FC23, GT24, ENSW25, aHW25}. It is therefore important to determine which graph algebras are $\cZ$-stable.
    \par A well-known combinatorial condition on directed graphs is that of Condition (K), which precludes the existence of subquotients stably isomorphic to $C(\bT)$, the $C^*$-algebra of continuous functions on the circle. Since algebras stably isomorphic to abelian $C^*$-algebras certainly cannot be $\cZ$-stable, and $\cZ$-stability passes to ideals and quotients (\cite{TW07}), Condition~(K) is necessary for $\cZ$-stability of $C^*(E)$. It is known that all such graph algebras have nuclear dimension at most two by \cite[Theorem~A]{FC23}. However, Condition~(K) does not prevent the existence of all elementary subquotients; any graph $E$ without cycles trivially satisfies Condition~(K), but the associated (approximately finite-dimensional) graph algebra $C^*(E)$ may have an elementary subquotient or even be elementary itself. Therefore, we introduce another combinatorial condition that prevents such subquotients from appearing: \textit{distinct detours} (Definition~\ref{detourcond}). One may think of this as a divisibility condition on the vertex projections of $C^*(E)$.
    \par Under one of two additional conditions, our first result characterizes $\cZ$-stability of $C^*(E)$ in terms of the combinatorics of $E$ and the (lack of) elementary subquotients of $C^*(E)$. In the first case, when $E$ has no cycles (and so $C^*(E)$ is an AF algebra), the result is known to experts. However, the existence of distinct detours in $E$ allows us to provide a new proof of this fact using graph-theoretic techniques and \cite[Theorem~1.2]{RT17}. We then may extend this result to the setting of graph algebras with finitely many ideals by forming a composition series for $C^*(E)$. These results are summarized in the following theorem:
	\begin{thm}\label{thmA}
		Suppose $E$ is a countable, row-finite graph that satisfies one of the following conditions:
        \begin{enumerate}[label={(\alph*)}, font=\upshape]
            \item $E$ has no cycles. \label{A1}
            \item $C^*(E)$ has finitely many ideals. \label{A2}
        \end{enumerate}
        Then the following are equivalent:
        \begin{enumerate}
            \item $E$ has distinct detours.\label{thmAdetours}
            \item $C^*(E)$ has no elementary subquotients.\label{thmAsubq}
            \item $C^*(E)\otimes \cZ \cong C^*(E)$.\label{thmAjiangsu}
        \end{enumerate}
	\end{thm}
    \noindent Note that, in either case, $E$ satisfies Condition~(K); trivially in case~\ref{A1} since $E$ has no cycles, and in case~\ref{A2} as $C^*(E)$ cannot have a $C(\bT)$ subquotient.
    Due to the restrictive nature of finite graphs, the condition on distinct detours in $E$ is equivalent to $E$ having no sources. As such, Theorem~\ref{thmA} reduces to the following statement:
    \begin{thm}\label{thmB}
	Let $E$ be a finite graph. Then the following are equivalent:
	\begin{enumerate}
		\item $E$ has Condition (K) and no sources. \label{thmBgraph}
		\item $C^*(E)$ has no elementary subquotients. \label{thmBsubq}
		\item $C^*(E)\otimes \cZ \cong C^*(E)$. \label{thmBjiangsu}
		\item $C^*(E)\otimes \cO_\infty \cong C^*(E)$. \label{thmBpureinf}
	\end{enumerate}
\end{thm}

We conjecture (Conjecture~\ref{zconj}) that $E$ having Condition~(K) and distinct detours is equivalent to $\cZ$-stability of $C^*(E)$. In particular, this would imply that finite nuclear dimension and $\cZ$-stability are equivalent for graph algebras without elementary subquotients. The following theorem, proved using the results of \cite{APTV24, TV23, TV24}, suggests that these are the correct combinatorial conditions for $\cZ$-stability. Furthermore, combined with \cite[Theorem~A]{FC23}, this verifies the Generalized Toms--Winter Conjecture (Conjecture~\ref{GTWC}, c.f. \cite[Question~D]{APTV24}) in the cases \ref{A1} and \ref{A2} of Theorem~\ref{thmA}.

    \begin{thm}\label{thmC}
        Let $E$ be a countable row-finite graph. Then the following are equivalent:
        \begin{enumerate}
            \item $E$ has Condition (K) and distinct detours. \label{thmCgraph}
            \item $C^*(E)$ has no elementary subquotients. \label{thmCsubq}
            \item $C^*(E)$ is pure. \label{thmCpure}
        \end{enumerate}
    \end{thm}
In Section~\ref{prelim}, we cover the necessary background on graph algebras and $\cZ$-stability. We proceed in Section~\ref{afgraph} to develop the notion of distinct detours and use this to prove Theorem~\ref{thmA}, case~\ref{A1}, using a technique of Schafhauser \cite[Proposition~6.1]{S15}. Finally, we use these results to prove case~\ref{A2} of Theorem~\ref{thmA} and Theorem~\ref{thmB} in Section~\ref{finideals}, and Theorem~\ref{thmC} in Section~\ref{puregraph}.
    
	\setcounter{thm}{0}
	\numberwithin{thm}{section}
	
	\begin{acknowledgments}
		The author would like to thank Christopher Schafhauser for many enlightening discussions about graph algebras and $\cZ$-stability. The author also thanks the anonymous referee for their helpful suggestions.
	\end{acknowledgments}

    \section{Preliminaries}\label{prelim}
    We begin by recalling some details of strongly self-absorbing $C^*$-algebras. Two $\ast$-homomorphisms $\phi,\psi \colon A \to B$ are \textit{approximately unitarily equivalent} (written $\phi \approx_u \psi$) if there is a sequence of unitaries $(u_n)$ in the multiplier algebra $M(B)$ so that $\lim_{n \to \infty}u_n^* \phi(a)u_n=\psi(a)$ for all $a \in A$. A unital $C^*$-algebra $\cD$ is \textit{strongly self-absorbing} when the first factor embedding of $\cD$ into $\cD \otimes_{\text{min}}\cD$ is approximately unitarily equivalent to an isomorphism, made precise in the following definition:
    \begin{defi}[{\cite[Definition~1.3(iv)]{TW07}}]
        A separable, unital $C^*$-algebra $\cD$ is strongly self-absorbing if $\cD \ncong \bC$ and there is an isomorphism $\phi\colon \cD \to \cD \otimes_{\text{min}} \cD$ satisfying $\phi \approx_u \text{id}_\cD \otimes 1_\cD$.
    \end{defi}
    Strong self-absorption is a very restrictive property on a $C^*$-algebra. By \cite[Theorem~1.7]{TW07}, any such $C^*$-algebra is simple and nuclear (so we may write $\otimes$ for the tensor product unambiguously) and is either purely infinite or stably finite with a unique tracial state (\cite[Theorem~1.8]{TW07}). In the presence of the Universal Coefficient Theorem of Rosenberg and Schochet (\cite{RS87}), the only strongly self-absorbing $C^*$-algebras are UHF-algebras of infinite type, the Cuntz algebras $\cO_2$ and $\cO_\infty$, the Jiang--Su algebra $\cZ$, and tensor products of $\cO_\infty$ with a UHF algebra of infinite type (\cite[Corollary~5.2, Corollary~5.7]{TW07}. 
    \par Important families of $C^*$-algebras are those which are \textit{$\cD$-stable} for some fixed strongly self-absorbing $C^*$-algebra $\cD$. A $C^*$-algebra $A$ is $\cD$-stable when $A \cong A\otimes \cD$ (\cite{TW07}). Importantly, these isomorphisms are always approximately unitarily equivalent to the first factor embedding of $A$ into $A \otimes \cD$ (\cite[Theorem~2.2]{TW07}), mirroring the isomorphism of $\cD$ and $\cD \otimes \cD$.
    \par Central sequence algebras are of great importance for understanding $\cD$-stability. We remind the reader of their definition. First, fix a free ultrafilter $\omega$ on $\bN$. For a unital $C^*$-algebra $A$, the algebra $A_\omega$ is defined to be the quotient of the $C^*$-algebra $\ell^\omega(A)$ of bounded sequences by the ideal $c_\omega(A)$ of sequences converging to zero along the ultrafilter $\omega$. Recall that $A \into A_\omega$ as constant sequences. The \textit{central sequence algebra} of $A$ is $A_\omega \cap A'$.
    \begin{thm}[{\cite[Theorem~2.2]{TW07}}] 
        Let $\cD$ be a strongly self-absorbing $C^*$-algebra, and let $A$ be a unital, separable $C^*$-algebra. Then $A$ is $\cD$-stable if and only if there is a unital embedding $\cD \into A_\omega \cap A'$.
    \end{thm}
    $\cD$-stability also enjoys many nice permanence properties, as illustrated in the following theorems of Toms and Winter.
    \begin{thm}[{\cite[Corollaries~3.1-4, Theorem~4.3]{TW07}}]\label{dstabperm}
    Let $A, B, I$, and $ (A_n)_{n \in \bN}$ be separable $C^*$-algebras.
        \begin{enumerate}
            \item If $A$ is a hereditary subalgebra of $B$, and $B$ is $\cD$-stable, then $A$ is $\cD$-stable.\label{dstabpermher}
            \item If $I\trianglelefteq A$, then $A$ is $\cD$-stable if and only if $I$ and $A/I$ are $\cD$-stable.\label{dstabpermideal}
            \item $A$ is $\cD$-stable if and only if $A \otimes M_n$ is $\cD$-stable if and only if $A \otimes \cK$ is $\cD$-stable.\label{dstabpermstable}
            \item If $A=\varinjlim A_n$ and $(A_n)_{n \in \bN}$ are $\cD$-stable, then $A$ is $\cD$-stable.\label{dstabpermlim}
        \end{enumerate}
    \end{thm}
    Of particular relevance to this paper is when $\cD=\cZ$, the Jiang--Su algebra. The Jiang--Su algebra $\cZ$ is a simple, unital, projectionless, nuclear $C^*$-algebra. We refer the reader to \cite{JS99} and \cite{S22} for details on the construction of $\cZ$.
    \par Nuclear dimension (\cite{WZ10}) is an important $C^*$-regularity condition and is a non-commutative analogue of topological covering dimension. We will not need the technicalities of nuclear dimension to prove our results, but its relationship with $\cZ$-stability is both a motivation for this work as well as important to the proof of Theorem~\ref{thmA}, case~\ref{A1}. The following conjecture illustrates this relationship.
    
    \begin{conj}[Toms--Winter Conjecture]
         Let $A$ be a separable, simple, nuclear, unital, infinite-dimensional $C^*$-algebra. Then the following are equivalent:
         \begin{enumerate}
             \item $A$ has finite nuclear dimension. \label{TW1}
             \item $A$ is $\cZ$-stable. \label{TW2}
             \item \label{TW3}$A$ has strict comparison. (see \cite{B88}, c.f. \cite{R04})
         \end{enumerate}
    \end{conj}
    \noindent By the results of \cite{CETWW21, CE20, W12, T14}, utilizing the work of \cite{MS14}, it is known that \ref{TW1} and \ref{TW2} are equivalent, and they both imply \ref{TW3}. Outside of the simple setting, the equivalence of the first two conditions no longer holds (an easy counterexample is $C(\bT)$). The proposed additional hypothesis is that of \textit{no elementary subquotients}. By \textit{subquotient} we mean a quotient of an ideal. A $C^*$-algebra is \textit{elementary} if it is isomorphic to $\cK(\cH)$ for some Hilbert space $\cH$.
    \begin{conj}[{Generalized Toms--Winter Conjecture (\cite[Question~D]{APTV24})}]\label{GTWC}
        Let $A$ be a separable, unital, nuclear $C^*$-algebra with no elementary subquotients. Then the following are equivalent:
        \begin{enumerate}
            \item $A$ has finite nuclear dimension.
            \item $A$ is $\cZ$-stable.
            \item $A$ is pure.
        \end{enumerate}
    \end{conj}
    Pure $C^*$-algebras were introduced by Winter in \cite{W12} (c.f. \cite{APTV24}). Pureness of a $C^*$-algebra is a regularity condition on the Cuntz semigroup $\text{Cu}(A)$ and may be thought of as ``$\cZ$-stability of the Cuntz semigroup," as every $\cZ$-stable $C^*$-algebra is pure (\cite[Proposition~5.2]{APTV24}). Some progress has been made in understanding this generalized form of the Toms--Winter conjecture, such as \cite{BL24, BGSW22, ENST20, APTV24}. Most relevant to us is the work of Robert and Tikuisis, where they proved the following theorem. Recall that an element $b \in B$ is called \textit{full} if $b$ generates $B$ as a closed, two-sided ideal. Similarly, a subset $X \subset B$ is full if $B$ is the smallest closed, two-sided ideal containing $X$.
    \begin{thm}[{\cite[Theorem~1.2]{RT17}}]
        Let $A$ be a separable, unital $C^*$-algebra with finite nuclear dimension. Then $A$ is $\cZ$-stable if and only if there are two full, orthogonal elements in $A_\omega \cap A'$.
    \end{thm}
    \par Having covered the necessary background on $\cZ$-stability, we proceed to discuss graph $C^*$-algebras. We refer the reader to Raeburn's book (\cite{R05}) for a thorough introduction to graph algebras. A directed graph is a four-tuple $E=(E^0,E^1,r,s)$, where $E^0$ is the set of \textit{vertices} of $E$, $E^1$ is the set of \textit{edges} of $E$, and $r,s \colon E^1\to E^0$ are the \textit{range} and \textit{source} maps. A vertex $v$ such that $r^{-1}(v)=\emptyset$ is called a \textit{source}, and a vertex with $s^{-1}(v)=\emptyset$ is a \textit{sink}. We say that $E$ is \textit{row-finite} if $r^{-1}(v)$ is finite for all vertices $v \in E^0$. We follow Raeburn's convention (\cite{R05}) and write paths in $E$ from right to left. Let $E^*$ denote the paths in $E$ of finite length. 
    \begin{defi}[{\cite[Proposition~4.1]{KPRR97}}]\label{CKdef}
        Given a directed graph $E$, a Cuntz--Krieger $E$-family $(S,P)$ is a family of orthogonal projections $\{P_v\}_{v \in E^0}$ and partial isometries $\{S_e\}_{e \in E^1}$ such that
        \begin{enumerate}
            \item $S_e^*S_e=P_{s(e)}$,
            \item $P_vS_e=P_v$ when $r(e)=v$, and \label{CK2}
            \item $P_v=\sum_{r(e)=v}S_eS_e^*$ when $0<\abs{r^{-1}(v)}<\infty$.
        \end{enumerate}
    \end{defi}
    \noindent The graph $C^*$-algebra $C^*(E)$ is the universal $C^*$-algebra generated by a Cuntz--Krieger family $(s,p)$. Given a path $\mu=\mu_1\mu_2\cdots\mu_n$, we write $s_\mu$ for the product $s_{\mu_1}s_{\mu_2}\cdots s_{\mu_n}$. Importantly, we can also define the following $\ast$-algebra that is dense in $C^*(E)$.
    \begin{defi}{\cite{GAP05, AMP07}}
        The Leavitt Path Algebra $L_\bC(E)$ is the universal $\ast$-algebra generated by a Cuntz--Krieger $E$-family $(s,p)$. Furthermore, $L_\bC(E)=\Span\{s_\mu s_\nu^*:\mu, \nu \in E^*, s(\mu)=s(\nu)\}$.
    \end{defi}
    \begin{thm}[{\cite[Corollary~1.16]{R05}}]
        $L_\bC(E)$ is a dense $\ast$-subalgebra of $C^*(E)$. That is
        $$C^*(E)=\clo{\Span}\{s_\mu s_\nu^*:\mu, \nu \in E^*, s(\mu)=s(\nu)\}\cong \clo{L_\bC(E)}$$
    \end{thm}
    Arithmetic with elements of $L_\bC(E)$ is very concrete due to the following theorem:
    \begin{thm}[{\cite[Corollary~1.15]{R05}}]
        Let $\mu, \nu, \gamma, \lambda$ be paths in $E^*$, and $E$ a row-finite graph. Then
        $$(s_\mu s_\nu^*)(s_\gamma s_\lambda^*)=\begin{cases}
            s_{\mu \gamma'}s_\lambda^* & \text{ if $\gamma = \nu\gamma'$}\\
            s_{\mu}s_{\lambda \nu'}^* & \text{ if $\nu = \gamma \nu'$}\\
            0 & \text{ otherwise.}
        \end{cases}$$
    \end{thm}
    Of note is the existence of a gauge action $\Gamma$ of the circle $\bT$ on $C^*(E)$ such that $\Gamma_z(p_v)=p_v$ and $\Gamma_z(s_e)=zs_e$ for all $z \in \bT,v \in E^0,$ and $e \in E^1$. Of particular interest are the ideals of $C^*(E)$ that are invariant under the gauge action. These ideals are parameterized by certain subsets of $E^0$. 
    \begin{defi}
        A subset $H \subseteq E^0$ is \textit{hereditary} if whenever $v \in H, w \in E^0$, and there is a path from $w$ to $v$, we have $w \in H$ as well. We call $H$ \textit{saturated} if for any vertex $v \in E^0$ with $\emptyset\neq r^{-1}(v)\subseteq H$, we have $v \in H$.
    \end{defi}
    \noindent Given a hereditary, saturated subset $H\subseteq E^0$, there are two important related graphs, $E_H$ and $E\setminus H$. The sets of edges and vertices are given as follows:
    \begin{align*}
        &E_H^0=H  & (E\setminus H)^0&=E^0\setminus H\\
        &E_H^1=\{e \in E^1: r(e)\in H\} &(E\setminus H)^1&=\{e \in E^1: s(e) \notin H\}
    \end{align*}
    The range and source maps of $E_H$ and $E\setminus H$ are the appropriate restrictions of $r$ and $s$ from $E$. The gauge-invariant ideals of $C^*(E)$ are then parameterized as follows:
    \begin{thm}[{\cite[Theorem~4.1]{BPRS00}}]
        Let $E$ be a row-finite graph. Then there is a bijection between hereditary, saturated subsets $H \subseteq E^0$ and gauge-invariant ideals $I \trianglelefteq C^*(E)$ given by $H \mapsto I_H$, where $I_H$ is the ideal of $C^*(E)$ generated by $\{p_v: v \in H\}$. Furthermore, there is a short exact sequence,
        \begin{center}
        \begin{tikzcd}
            0 \arrow[r] & I_H \arrow[r] & C^*(E) \arrow[r] & C^*(E \setminus H) \arrow[r] & 0,
        \end{tikzcd}
        \end{center}
        with $I_H$ containing $C^*(E_H)$ as a full corner.
    \end{thm}
    Furthermore, in the presence of Condition~(K), every ideal of $C^*(E)$ will be gauge invariant. Given a vertex $v \in E^0$, a \textit{return path} for $v$ is a path $\mu \in E^*$ that only contains $v$ as its source and range.
    \begin{defi}[{\cite[Section~6]{KPRR97}}]
        A graph $E$ is said to have Condition~(K) if every vertex on a cycle has two distinct return paths.
    \end{defi}
    \noindent Note that graphs without cycles trivially satisfy Condition~(K). 
    \begin{thm}[{\cite[Theorem~4.4]{BPRS00}, c.f. \cite[Theorem~6.6]{KPRR97}}]
        Let $E$ be a row-finite graph with Condition~(K). Then every ideal of $C^*(E)$ is gauge-invariant (and so given by a hereditary, saturated subset of $E^0$).
    \end{thm}
    While many of the above results have considered row-finite graphs, we are also interested in graphs with infinite receivers. The standard construction to reduce a problem of arbitrary graphs to that of row-finite graphs is a \textit{Drinen--Tomforde desingularization}, defined in \cite{DT05}.
    \begin{defi}[{\cite[Chapter~5]{R05}}]
        An infinite path $\mu=\mu_1\mu_2\cdots$ is said to be \textit{collapsible} if the following hold:
        \begin{enumerate}
            \item If $\mu$ has an exit, it occurs at $r(\mu)$.
            \item The set $r^{-1}(r(\mu_i))$ is finite for every $i$.
            \item We have $r^{-1}(r(\mu))$ is $\{\mu_1\}$.
        \end{enumerate}
    \end{defi}
    Given a row-finite graph $E$ and a collapsible path $\mu$ in $E$, a new graph may be formed by collapsing $\mu$ into a single vertex and redefining the source and range maps appropriately (for more details, see \cite[Chapter~5]{R05}).
    \begin{defi}[{\cite[Definition~2.2]{DT05}}]
        Let $E$ be an arbitrary graph. A \textit{Drinen--Tomforde desingularization} $(F,M)$ of $E$ is a row-finite graph $F$ with a collection of collapsible paths $M$ such that collapsing the paths in $M$ results in the graph $E$.
    \end{defi}
    \noindent With this definition in hand, the following theorem (\cite[Theorem~2.11]{DT05}) provides the reduction of many problems regarding arbitrary graphs to that of row-finite graphs, especially when combined with \cite[Theorem~2.8]{B77} to produce a stable isomorphism.
    \begin{thm}[{\cite[Theorem~2.11]{DT05}}]
        Let $E$ be an arbitrary graph and $(F,M)$ a Drinen--Tomforde desingularization. Then $C^*(E)$ is isomorphic to a full corner of $C^*(F)$.
    \end{thm}
    \par In our proof of Theorem~\ref{thmA}, case \ref{A1}, we will need to ensure our inclusion of subgraphs corresponds to inclusion of graph $C^*$-algebras. We will use Jeong and Park's technique from \cite{JP02}, adapted for our convention on paths.
    \begin{defi}[{\cite[Definition~3.2]{JP02}}]
        Let $F$ be a subgraph of $E$. $F$ is \textit{entrance-complete} if whenever $f \in F^1$ and $e \in E^1$ with $r(e)=r(f)$, then $e \in F^1$ and $s(e) \in F^0$.
    \end{defi}
    \noindent This essentially forces $F$ to include all edges entering vertices of $F$ that are not sources in $F$. An inclusion of an entrance-complete subgraph corresponds to an inclusion of the graph $C^*$-algebras. We remark that this result also extends to the Leavitt path algebras $L_\bC(F)$ and $L_\bC(E)$.
    \begin{thm}[{\cite[Proposition~3.1]{JP02}}]
        Let $E$ be a row-finite directed graph and $F$ an entrance-complete subgraph of $E$. Let $B\subseteq C^*(E)$ be the $C^*$-subalgebra of $C^*(E)$ generated by $\{s_e,p_v: e \in F^1, v \in F^0\}$. Then $C^*(F)$ is isomorphic to $B$ by sending the Cuntz--Krieger family for $C^*(F)$ to the family $(s,p)$ defined by $\{s_e,p_v:e \in F^1, v \in F^0\}$.
    \end{thm}
    
    \par Finally, we lay out some terminology regarding paths and trees. Let $E^{\leq\infty}$ denote the set of paths in $E$ that are either infinite backwards paths (e.g., $\mu_1\mu_2\mu_3\cdots$) or are finite paths starting at a source of $E$. A path $\mu \in E^{\leq\infty}$ is \textit{simple} if every vertex along $\mu$ only occurs once on $\mu$. This brings us to the following definition:
    \begin{defi}[{c.f. \cite[Section~1.1]{H01}}]
        Let $\mu \in E^{\leq\infty}$. A path $\nu\in E^*$ is a \textit{detour} for $\mu$ if $s(\nu)$ and $r(\nu)$ lie on $\mu$.
    \end{defi}
    \noindent Note that finite subpaths of $\mu$ are detours for $\nu$. We will say that a detour $\nu$ is \textit{distinct} if it contains an edge that is not contained in $\mu$.
    \par In-trees will also play a major role in our proof, so we remind the reader of their definition below. Recall that a \textit{rooted tree} is a tree with a distinguished vertex $v$ such that every other vertex has a path to $v$ on the underlying undirected graph.
    \begin{defi}[{\cite[Page~207]{N74}}]
	   A rooted, directed tree $(F,v)$ is an \textit{in-tree} if every directed edge has an orientation pointing towards $v$.
    \end{defi}

\section{\texorpdfstring{$\cZ$}{Z}-stable AF Graph Algebras}\label{afgraph}

We begin with some general results regarding $\cD$-stability, for any strongly self-absorbing $C^*$-algebra $\cD$. The following proposition is one of our key tools for understanding $\cD$-stability for infinite graphs, allowing us to reduce the problem to the setting of unital $C^*$-algebras. The proof is analogous to that of \cite[Proposition~6.1]{S15}.
\begin{prop}\label{corners}
	Suppose $A$ is a $C^*$-algebra and $(p_n)_{n \in \bN}\subset M(A)$ is a countable collection of pairwise orthogonal projections such that $\sum_{n \in \bN} p_n=1$, where convergence is taken in the strict topology. Then, for any strongly self-absorbing $C^*$-algebra $\cD$, $A$ is $\cD$-stable if and only if $p_nAp_n$ is $\cD$-stable for all $n \in \bN$.
\end{prop}

\begin{proof}
	The forward direction is obvious since $\cD$-stability passes to hereditary subalgebras by Theorem~\ref{dstabperm}\ref{dstabpermher}. To prove the reverse direction, we first consider the case where $p,q \in M(A)$ are projections such that $p+q=1$ and both $pAp$ and $qAq$ are $\cD$-stable. Then $ApA$ is stably isomorphic to $pAp$ by \cite[Theorem~2.8]{B77}, as $pAp$ is full in $ApA$. Similarly, $AqA$ and $qAq$ are stably isomorphic. Hence $ApA$ and $AqA$ are $\cD$-stable, as $\cD$-stability is preserved by stable isomorphism by Theorem~\ref{dstabperm}\ref{dstabpermstable}. As in \cite{S15}, we show that $ApA + AqA =A$. Fix an approximate unit $(h_i)_{i \in I} \subset A$. For each $a \in A$, observe that
	$$ \lim_{i \in I} (h_ipa + h_iqa)=pa+qa=(p+q)a = a$$
	Since the terms in the net $h_ipa +h_iqa$ belong to $ApA + AqA$, we conclude that $a \in ApA + AqA$, and so $ApA+AqA=A$. 
	\par Now consider the sequence
	\begin{center}
		\begin{tikzcd}
			0 \arrow[r] & ApA \arrow[r] & A \arrow[r] & AqA/(AqA \cap ApA) \arrow[r] & 0
		\end{tikzcd}
	\end{center}
	which is exact since $A/ApA =(ApA + AqA)/ApA \cong AqA/(AqA \cap ApA)$. The quotient \mbox{$AqA/(AqA \cap ApA)$} is $\cD$-stable, as it is a quotient of the $\cD$-stable algebra $AqA$ by Theorem~\ref{dstabperm}\ref{dstabpermideal}. We conclude that $A$ is $\cD$-stable since $\cD$-stability is preserved when taking extensions, also by Theorem~\ref{dstabperm}\ref{dstabpermideal}.
	\par We now return to the original case. Define $q_n$ to be the partial sum $q_n\coloneq p_1 + \cdots + p_n$. By assumption, $q_n \to 1$ strictly in $M(A)$, so we have that $q_naq_n \to a$ for all $a \in A$. This allows us to conclude that $\varinjlim q_nAq_n = A$. We see that $q_1Aq_1=p_1Ap_1$ is $\cD$-stable by assumption. For our inductive step, suppose that $q_nAq_n$ is $\cD$-stable. Since $q_{n+1}=q_n+p_{n+1}$, and both $q_nAq_n$ and $p_{n+1}Ap_{n+1}$ are $\cD$-stable, it follows that $q_{n+1}Aq_{n+1}$ is $\cD$-stable by the first half of this proof, since $q_{n+1}$ is the identity of $q_{n+1}Aq_{n+1}$. As direct limits preserve $\cD$-stability by Theorem~\ref{dstabperm}\ref{dstabpermlim}, we conclude that $A$ is $\cD$-stable.
\end{proof}

The following corollary applies Proposition~\ref{corners} to graph $C^*$-algebras.

\begin{cor}\label{graphcorners}
	Let $\cD$ be a strongly self-absorbing $C^*$-algebra, and let $E$ be a graph with countably many vertices. Then $C^*(E)$ is $\cD$-stable if and only if $p_vC^*(E)p_v$ is $\cD$-stable for all $v \in E^0$.
\end{cor}

\begin{proof}
	Apply Proposition~\ref{corners} to the countable collection of pairwise orthogonal projections $(p_v)_{v \in E^0}$.
\end{proof}

The following definition is our main hypothesis for $\cZ$-stability. Conceptually, one can think of this as a divisibility condition on vertex projections in $C^*(E)$. 

\begin{defi}\label{detourcond}
	We say a row-finite graph $E$ \textit{has distinct detours} if every simple path $\mu \in E^{\leq\infty}$ has a distinct detour $\nu \in E^*$.
\end{defi}

\begin{rmk}\label{detournosources}
    Observe that any row-finite graph $E$ having distinct detours necessarily does not have any sources. For, if $v \in E^0$ is a source, then $v$ is a path of length $0$ in $E^{\leq\infty}$. But since $v$ is a source, this path is the unique path from $v$ to $v$. Thus no distinct detour can exist.
\end{rmk}

We first wish to show that this property is necessary for $C^*(E)$ to have no elementary subquotients. As a prerequisite, we need to know that the presence of elementary subquotients is preserved by stable isomorphism.

\begin{lem}\label{stabsubquo}
    Let $A$ be a $C^*$-algebra. Then $A$ has an elementary subquotient if and only if $A \otimes \cK$ has an elementary subquotient.
\end{lem}

\begin{proof}
    We begin with the forward direction. Note that if $I\trianglelefteq A$, then $I \otimes \cK \trianglelefteq A \otimes \cK$. Thus, if $B\cong I/J$ is an elementary subquotient of $A$, then $B \otimes \cK\cong (I/J)\otimes \cK\cong (I \otimes \cK)/(J\otimes \cK)$ is an elementary subquotient of $A \otimes \cK$. 
    \par For the reverse direction, note that by Corollary 9.4.6 of \cite{BO08}, ideals of $A \otimes \cK$ have the form $I \otimes \cK$ for an ideal $I \trianglelefteq A$. Therefore, a subquotient of $A \otimes \cK$ has the form $(I/J)\otimes \cK$ for ideals $J \trianglelefteq I\trianglelefteq A$. Thus, if $(I/J)\otimes \cK$ is an elementary $C^*$-algebra, then $I/J$ is Morita equivalent (\cite[Definition~6.10]{R74}) to $\bC$, whence $I/J \cong \cK(\cH)$ for some Hilbert space $\cH$.
\end{proof}

We can now prove that that lacking elementary subquotients implies the existence of distinct detours.

\begin{lem}\label{detournec}
    Suppose $E$ is a row-finite graph such that $C^*(E)$ has no elementary subquotients. Then $E$ has distinct detours.
\end{lem}

\begin{proof}
    We prove the contrapositive. Suppose that $\mu\in E^{\leq\infty}$ is a simple path without a distinct detour. Note that subpaths of $\mu$ are the unique paths connecting any pair of vertices on $\mu$. Let $H\subset E^0$ be the smallest hereditary saturated subset containing the vertices of $\mu$. The corresponding graph $E_H$ contains $\mu$ as a path, and $C^*(E_H)$ is stably isomorphic to the ideal $I_H$. Now, define $S$ to be the set of vertices in $(E_H)^0\setminus \mu$ that have a path to a vertex on $\mu$. Since $\mu$ does not have a distinct detour in $E$, it also does not have one in $E_H$. Therefore, $S$ must be hereditary (else, any hereditary subset containing $S$ would intersect $\mu$). Let $K$ be the smallest saturated set containing $S$, which is also hereditary by \cite[Remark~4.11]{R05}. Note that $K$ again does not contain any vertices of $\mu$ as $E_H^0\setminus \mu$ is saturated and contains $S$. We then have $C^*(E_H\setminus K) \cong C^*(E_H)/I_K$, and $E_H\setminus K$ contains $\mu$. Let $p \in M(C^*(E_H\setminus K))$ be the sum of the vertex projections on $\mu$, where convergence is taken in the strict topology. By the construction of $K$, no vertices on $\mu$ receive edges from vertices that are not also on $\mu$. We then have $pC^*(E_H\setminus K)p \cong M_{\abs{\mu}}$, with the convention that $M_\infty = \cK$. By \cite[Theorem~2.8]{B77}, $pC^*(E_H\setminus K)p$ is stably isomorphic to $C^*(E_H\setminus K)$. Therefore, $C^*(E_H\setminus K)$ is an elementary $C^*$-algebra, which, in particular, is a subquotient of $C^*(E_H)$. By Lemma~\ref{stabsubquo}, it follows that $I_H\otimes \cK$ has an elementary subquotient, as does $I_H$. As ideals of $I_H$ are also ideals of $C^*(E)$, we can finally conclude that $C^*(E)$ has an elementary subquotient.
\end{proof}

This also leads to the following easy corollary, using Theorem~\ref{dstabperm}\ref{dstabpermideal} and the fact that elementary $C^*$-algebras are not $\cD$-stable.

\begin{cor}
    Let $\cD$ be a strongly self-absorbing $C^*$-algebra, and suppose $E$ is a row-finite graph such that $C^*(E)\otimes \cD \cong C^*(E)$. Then $E$ has distinct detours.
\end{cor}

We will now investigate this property in the case of AF graph algebras. In order to apply \cite[Theorem~1.2]{RT17}, we will use inclusions of in-trees to build systems of approximately central matrix units. We will need the subalgebra generated by the matrix units to be unital, which is proven by the following lemma. This lemma is well-known, but a proof is included here for convenience.

\begin{lem}\label{unitalmatrix}
    Let $(F,v)$ be a finite in-tree. Let $\Lambda_v$ denote the set of paths from a source of $F$ to $v$. Then in $C^*(F)$, we have 
    $$p_v=\sum_{\lambda \in \Lambda_v} s_\lambda s_\lambda^*$$
\end{lem}

\begin{proof}
    Since $F$ is a finite graph, the lengths of paths in $F$ are uniformly bounded by some maximum length. We therefore induct on the maximum lengths of paths. In the base case, we consider an in-tree such that each path has length at most one. Then the result is simply Definition~\ref{CKdef}\ref{CK2}. For our inductive step, suppose the result holds for any finite in-tree whose paths are at most length $n$. Suppose $(F,v)$ is a finite in-tree such that all paths have length at most $n+1$. Let $u_1, u_2, \dots, u_m$ be the sources of $F$, and set $G\coloneq F \setminus\{ u_1,u_2,\dots,u_m\}$. Then $(G, v)$ is a finite in-tree such that all paths have length at most $n$. Using the inductive hypothesis and applying Definition~\ref{CKdef}\ref{CK2} to the sources of $G$ yields the desired result.
\end{proof}

When we construct our approximately central matrix units, we need to avoid one-dimensional summands. These summands are avoided whenever our inclusion of in-trees satisfies the following definition:

\begin{defi}
	Let $(F,v) \subset (F',v)$ be an inclusion of finite in-trees. We say that $(F,v) \subset (F',v)$ is \textit{nondegenerate} if for all sources $u \in F$ and $w\in F'$, the number of paths from $w$ to $u$ that do not otherwise intersect $F$ belongs to the set $\{0, 2, 3, \dots\}$. Otherwise, we say that $(F,v) \subset (F',v)$ is \textit{degenerate}.
\end{defi}

\begin{rmk}\label{nondegensources}
    Observe that, if $(F,v)\subset (F',v)$ is a nondegenerate inclusion, $F$ and $F'$ have no common sources. If $w \in F^0$ is a source in both $F$ and $F'$, then there is a unique path in $F'$ from $w$ (as a source of $F'$) to $w$ (as a source of $F$), contradicting the nondegeneracy of the inclusion.
\end{rmk}

The following lemma will be important for understanding the structure of paths between sources in nondegenerate conclusions.

\begin{lem}\label{Nijnot0}
    Let $(F,v)\subset (F',v)$ be a nondegenerate inclusion, and let $u$ be a source of $F'$. Then there is a source $w \in F^0$ and paths $\mu$ and $\nu$ from $u$ to $w$ such that $\mu$ and $\nu$ only intersect $F$ at $r(\mu)=r(\nu)=w$.
\end{lem}

\begin{proof}
    Fix a source $w_1 \in F^0$ and a path $\nu_1$ from $u$ to $w_1$. If there are paths $\mu$ and $\nu$ from $u$ to $w_1$ satisfying the lemma, then we are done. Otherwise, because the inclusion $(F,v)\subset (F',v)$ is nondegenerate, we can find a second path $\mu_1\neq \nu_1$ from $u$ to $w_1$. As $w_1$ does not satisfy the conclusion of the lemma, at least one of $\nu_1$ and $\mu_1$ will reach a vertex of $F$ other than $w$; without loss of generality, suppose $\mu_1$ has this property. Let $w_2$ be the first vertex of $F$ reached by $\mu_1$. Note that $w_2$ is necessarily a source of $F$ since it is the first vertex of $F$ on $\mu_1$. If there are paths to $w_2$ satisfying the lemma, we are done; else, factor the path $\mu_1$ as $\gamma_1\nu_2$, where $\gamma_1$ is a path from $w_2$ to $w_1$ and $\nu_2$ is from $u$ to $w_2$. Now repeat the above argument at $w_2$ using the path $\nu_2$. Proceeding inductively, we produce sources $w_1,w_2,\dots,w_n \in F$ and paths $\gamma_i$ from $w_{i+1}$ to $w_i$ for $1\leq i <n$. This process must terminate because $F$ is a finite tree, and so has finitely many sources and no cycles. Thus we obtain a source $w_k$ in $F^0$ satisfying the conclusion of the lemma.
\end{proof}

As we will want our system of approximately computing matrix units to lack one-dimensional representations, we need to be able to find a nondegenerate inclusion for any given in-tree.

\begin{lem}\label{nondegenincl}
	Let $E$ be a row-finite graph with no cycles that has distinct detours. Let $(F,v)$ be a finite in-tree contained in $E$. Then there exists an entrance-complete finite in-tree $(F',v)$ contained in $E$ such that $(F,v) \subset (F', v)$ is a nondegenerate inclusion.
\end{lem}

\begin{proof}
	By Remark~\ref{detournosources}, $E$ is an infinite graph since it has neither sources nor cycles. We inductively define a sequence of nested in-trees as follows. Set $(F_0,v)=(F,v)$. If we are given $(F_n,v)$, we define $(F_{n+1},v)$ as follows: Set $F_{n+1}^0=F_n^0 \cup s(r^{-1}(F_n^0))$ and $F_{n+1}^1=F_n^1 \cup r^{-1}(F_n^0)$, with the range and source maps for $F_{n+1}$ coming from $E$. It is clear from construction that $(F_n, v)_{n>0}$ is a sequence of entrance-complete in-trees. Furthermore, since $E$ has infinitely many vertices, none of which are sources (by Remark~\ref{detournosources}), each of these inclusions is proper. By way of contradiction, suppose that $(F,v)=(F_0,v) \subset (F_n,v)$ is a degenerate inclusion for all $n\geq1$. We claim we may find a sequence of sources $u_n \in F_n$ such that $u_0,u_1,u_2,\dots$ defines a path $\mu \in E^{\leq\infty}$ and each $u_n$ has a unique path to $u_0$ (namely, a subpath of $\mu$). 
	\par By our assumption that each inclusion $(F_0,v)\subset (F_n,v)$ is degenerate, it is clear that we may find at least one source in each $F_n$ ($n\geq 1$) with a unique path to a source in $F_0$.  Let $S \subset E^0$ be the set of all vertices that are a source in some $F_n$ and have a unique path to a source of $F_0$, noting $\abs{S}=\infty$. Letting $w_1, w_2, \dots, w_m$ be the sources of $F_0$, set $S_{w_i} \subset S$ to be the sources with a unique path to $w_i$. As $S=S_{w_1} \cup \cdots \cup S_{w_m}$, it follows that $S_{w_j}$ is infinite for some $j$. Set $u_0=w_j$. Now, repeat this process, but replace $S$ with $S_1$, the sources with a unique path to a vertex in $s(r^{-1}(u_0))$. Proceeding inductively in this manner produces a sequence of sources $(u_n)$. Furthermore, since $u_n \in s(r^{-1}(u_{n-1}))$, this sequence of vertices defines a path in $E^{\leq\infty}$ as desired (in fact, the path defined is unique since the vertices were constructed from degenerate inclusions).
	\par As $E$ has distinct detours, we may find a distinct detour $\nu$ of $\mu=\mu_1\mu_2\cdots$. Suppose that $s(\nu)=u_l$ and $r(\nu)=u_k$. Note that $k<l$ as $E$ has no cycles. We then see that $u_{l}$ has two distinct paths to $u_0$; namely $\mu_1\mu_2\cdots \mu_l$ and $\mu_1\mu_2\cdots \mu_k \nu$, which is a contradiction. We conclude that we may find an $N \in \bN$ such that $(F,v)\subset (F_N,v)$ is nondegenerate, and we set $F'=F_N$.
\end{proof}

We now have the necessary tools to build the desired system of matrix units that commute with $L_\bC(F)$ for a finite in-tree $(F,v)$.

\begin{lem}\label{centralmatrix}
    Let $E$ be a row-finite graph, and $(F,v)\subset E$ an entrance-complete finite in-tree. If $(F',v)\subset E$ is another entrance-complete finite in-tree such that $(F,v)\subset (F',v)$ is a nondegenerate inclusion, then we may construct a unital $\ast$-homomorphism $\phi\colon M_2\oplus M_3\to p_v L_\bC(F')p_v$ such that the range of $\phi$ commutes with $p_vL_\bC(F)p_v$.
\end{lem}

\begin{proof}
    Note that as $(F,v)$ is entrance complete, we may regard $L_\bC(F)$ as a subalgebra of $L_\bC(F')$ by \cite[Proposition~3.1]{JP02}, and thus $p_vL_\bC(F)p_v\subset p_vL_\bC(F')p_v$ is a unital inclusion. Now, let $\{u_1,u_2,\dots, u_n\}$ denote the sources of $F$ and $\{v_1,v_2,\dots, v_m\}$ the sources of $F'$. Recall that these two sets are disjoint by Remark~\ref{nondegensources}. Let $\Lambda_i$ denote the set of paths from $u_i$ to $v$. Then for each pair of paths $\mu,\nu$ from some $v_j$ to some $u_i$ such that $u_i$ is the only vertex of $F$ on $\mu$ and $\nu$, we define
    $$T_\mu T_\nu^*=\sum_{\lambda \in \Lambda_i}s_\lambda s_\mu s_\nu^* s_\lambda^*.$$
    We claim that the set 
    $$\mathbb{M}\coloneq \Big\{T_\mu T_\nu^*:s(\mu)=s(\nu)=v_j, r(\mu)=r(\nu)=u_i, \mu \cap F = \nu \cap F = \{u_i\}\Big\}$$
    forms a system of matrix units. Given two paths $\mu, \nu$ originating at sources of $F'$ with $r(\mu)=r(\nu)=v$, observe that $s_\nu^*s_\mu=p_{s(\mu)}$ if $\mu=\nu$ and $s_\nu^* s_\mu=0$ otherwise. Furthermore, paths $\gamma$ from a source in $F'$ to $v$ uniquely factorize as $\lambda \mu$ for some paths $\lambda$ and $\mu$ so that $\mu \cap F=r(\mu)$. Thus, when $\nu=\mu'$ with $r(\nu)=r(\mu')=u_i$, we have
    \begin{equation*}
    \begin{split}
		(T_\mu T_\nu^*)(T_{\mu'}T_{\nu'}^*)&=  \left(\sum_{\lambda \in \Lambda_i}s_\lambda s_\mu  s_\nu^*s_\lambda^*\right) \left( \sum_{\lambda' \in \Lambda_i}s_{\lambda'} s_{\mu'}  s_{\nu'}^*s_{\lambda'}^*\right) \\
        &=  \left(\sum_{\lambda \in \Lambda_i}(s_\lambda s_\mu  s_\nu^*s_\lambda^*)(s_{\lambda} s_{\mu'}  s_{\nu'}^*s_{\lambda}^*)\right) \\
		&= \sum_{\lambda \in \Lambda_i}s_\lambda s_\mu  s_{\nu'}^*s_\lambda^*\\
		&= T_\mu T_{\nu'}^*
    \end{split}
	\end{equation*}
    and this product is zero if $\nu \neq \mu'$. By Lemma~\ref{unitalmatrix}, we have that the sum of the diagonal matrix units $T_\mu T_\mu^*$ is $p_v$, the unit of $p_vL_\bC(F')p_v$. Let $N_{i,j}$ denote the number of paths from $v_j$ to $u_i$ such that $u_i$ is the only vertex of $F$ on each path. We observe that
	$$C^*(\mathbb{M}) \cong\bigoplus_{(i,j)=(1,1)}^{(i,j)=(n,m)}M_{N_{i,j}}$$
    Note that, since $(F,v) \subset (F',v)$ is nondegenerate, we have that $N_{i,j}\neq 1$ for all $i$ and $j$. Furthermore, at least one $N_{i,j}$ is not zero by Lemma~\ref{Nijnot0}. Finally, suppose we are given paths $\gamma$ and $\eta$ in $F$ with $r(\gamma)=r(\eta)=v$ and $s(\gamma)=s(\eta)=u_k$. Let $\beta_k$ be the set of paths from $u_i$ to $u_k$. We see that
    \begin{equation*}
    \begin{split}
        T_\mu T_\nu^* s_\gamma s_\eta^* &=\left(\sum_{\lambda \in \Lambda_i}s_\lambda s_\mu s_\nu^* s_\lambda^*\right)s_\gamma s_\eta^*\\
        &=\left(\sum_{\alpha \in \beta_k}s_\gamma s_\alpha s_\mu s_\nu^* s_\alpha^*s_\gamma^*\right)s_\gamma s_\eta^*\\
        &=\sum_{\alpha \in \beta_k}s_\gamma s_\alpha s_\mu s_\nu^* s_\alpha^*s_\eta^*\\
        &=\sum_{\alpha \in \beta_k}(s_\gamma s_\eta^*) s_\eta s_\alpha s_\mu s_\nu^* s_\alpha^*s_\eta^*\\
        &=s_\gamma s_\eta^*\left(\sum_{\alpha \in \beta_k} s_\eta s_\alpha s_\mu s_\nu^* s_\alpha^*s_\eta^*\right)\\
        &= s_\gamma s_\eta^* \left(\sum_{\lambda \in \Lambda_i}s_\lambda s_\mu s_\nu^* s_\lambda^*\right)\\
        &= s_\gamma s_\eta^* T_\mu T_\nu^* 
    \end{split}
    \end{equation*}
    As every element of $p_vL_\bC(F)p_v$ may be written as a linear combination of elements of the form $s_\gamma s_\eta^*$ as above by repeated applications of Definition~\ref{CKdef}\ref{CK2}, we conclude that $C^*(\mathbb{M})$ commutes with $p_vL_\bC(F)p_v$. Finally, by writing $N_{i,j}=2x_{i,j}+3y_{i,j}$ for nonnegative integers $x_{i,j}$ and $y_{i,j}$, we obtain a unital $\ast$-homomorphism $\phi\colon M_2 \oplus M_3\to C^*(\mathbb{M})\subset p_vL_\bC(F')p_v$ by inserting $x_{i,j}$ copies of $M_2$ and $y_{i,j}$ copies of $M_3$ as block diagonal matrices inside $M_{N_{i,j}}$.
\end{proof}

We finally have all of the necessary tools in hand to prove case \ref{A1} of Theorem~\ref{thmA}.

\begin{proof}[Proof of Theorem~\ref{thmA}, case \ref{A1}]
    It is clear that \ref{thmAsubq} follows from \ref{thmAjiangsu} since $\cZ$-stability is preserved by taking ideals and quotients by Theorem~\ref{dstabperm}\ref{dstabpermideal}, and \ref{thmAsubq} implies \ref{thmAdetours} was proven in Lemma~\ref{detournec}. Thus we only need to show that the existence of distinct detours implies $\cZ$-stability. By Corollary~\ref{graphcorners}, it is enough to show that $p_vC^*(E)p_v$ is $\cZ$-stable for all $v \in E^0$. $C^*(E)$ is separable since $E$ is countable, so we may fix a countable dense sequence $\{x_1,x_2,\dots\} \subset p_vC^*(E)p_v$. Define a sequence of finite subsets by first defining $\cF_1=\{x_1\}$ and then proceeding inductively by defining $\cF_{n+1}=\cF_n \cup \{x_{n+1}\}$. For each $n \in \bN$, we may find an entrance-complete finite in-tree $(F_n,v) \subset E$ such that the distance between $\cF_n$ and $p_vL_\bC(F_n)p_v$ is less than $\frac{1}{2n}$. By applying Lemma~\ref{nondegenincl}, we find an entrance-complete finite in-tree $(F_n',v)$ such that $(F_n,v) \subset (F_n',v)$ is a nondegenerate inclusion for each $n \in \bN$. Applying Lemma~\ref{centralmatrix} to each inclusion $(F_n,v) \subset (F_n',v)$ produces a sequence of unital $\ast$-homomorphisms $\phi_n\colon M_2 \oplus M_3 \to p_vL_\bC(F_n')p_v\subset p_vC^*(E)p_v$ such that the range of $\phi_n$ commutes with $p_vL_\bC(F_n)p_v$. Since any element of $\cF_n$ is within $\frac{1}{2n}$ of $p_vL_\bC(F_n)p_v$, it follows that the range of $\phi_n$ commutes with $\cF_n$ up to a tolerance of $\frac{1}{n}$.
    \par Now fix a free ultrafilter $\omega$ on $\bN$, and define a $\ast$-homomorphism
    $$\Phi\colon M_2 \oplus M_3 \to (p_vC^*(E)p_v)_\omega \cap (p_vC^*(E)p_v)'$$
    $$\Phi(y)=(\phi_n(y))_{n \in \bN}$$
    for each $y \in M_2 \oplus M_3$. Note that the elements $y_1=(e_{11},e_{11})$ and $y_2=(e_{22}, e_{22}+e_{33})$ are orthogonal in $M_2 \oplus M_3$. Furthermore, they are both full in $M_2 \oplus M_3$. Since $\Phi$ is a unital $\ast$-homomorphism, it follows that $\Phi(y_1)$ and $\Phi(y_2)$ are full, orthogonal elements in $(p_vC^*(E)p_v)_\omega \cap (p_vC^*(E)p_v)'$, and so it follows that $p_vC^*(E)p_v$ is $\cZ$-stable by \cite[Theorem~1.2]{RT17}.
\end{proof}

We obtain the known-to-experts characterization of $\cZ$-stability for unital AF-algebras as a corollary to Theorem~\ref{thmA}, case \ref{A1}.

\begin{cor}
    Let $A$ be a separable, unital AF-algebra. Then $A \otimes \cZ\cong A$ if and only if $A$ has no elementary subquotients.
\end{cor}

\begin{proof}
    Only the reverse direction has meaningful content, as the forward direction is clear. Given a unital AF-algebra $A$, by \cite[Proposition~2.12]{R05} and \cite[Theorem~2.8]{B77}, $A \otimes \cK \cong C^*(E)\otimes \cK$, where $E$ is a Bratteli diagram for $A$. Since $A$ does not have any elementary subquotients, neither does $C^*(E)$ by Lemma~\ref{stabsubquo}. From Theorem~\ref{thmA}, case \ref{A1}, we have $C^*(E)\otimes \cZ \cong C^*(E)$, and hence $A$ is $\cZ$-stable by Theorem~\ref{dstabperm}\ref{dstabpermstable}.
\end{proof}

We may be curious about the assumption of row-finiteness for our graphs. In the case where our graphs do not have cycles, this is a necessary condition for $\cZ$-stability.

\begin{cor}\label{zstabAF}
	Let $E$ be a countable graph with no cycles. If $C^*(E)\cong C^*(E)\otimes \cZ$, then $E$ is row-finite.
\end{cor}

\begin{rmk}
	Corollary~\ref{zstabAF} fails when the graph has cycles. For instance, the graph $F$ with a single vertex and countably infinitely many edges is not row-finite, but $C^*(F)\cong \cO_\infty$, which is $\cZ$-stable by \cite[Theorem~5]{JS99}.
\end{rmk}

\begin{proof}[Proof of Corollary~\ref{zstabAF}]
	We prove the contrapositive. Suppose that $E$ is not row-finite, and let $(F,M)$ be a Drinen--Tomforde desingularization of $E$. Let $\mu=\mu_1\mu_2\cdots$ be a collapsible path in $M$. Note that $\mu$ is a simple path in $F^{\leq \infty}$; we claim that it does not have a distinct detour. Any distinct detour $\nu$ of $\mu$ contains an edge not on $\mu$, and $r(\mu)$ is the only exit of $\mu$. Thus, $\nu$ must pass through $r(\mu)$, so we may as well choose $\nu$ so that $s(\nu)=r(\mu)$. We know $r(\nu)$ lies on $\mu$, so let $s(\mu_n)=r(\nu)$. Then the path $\gamma=\nu\mu_1\mu_2\cdots\mu_n$ is a cycle with $s(\gamma)=r(\gamma)=r(\nu)$. However, $C^*(F)$ is AF if and only if $F$ does not have cycles. Because $C^*(F)$ and $C^*(E)$ are stably isomorphic, and $C^*(E)$ is AF because $E$ does not have cycles, we conclude that $C^*(F)$ is AF as well. Therefore, $F$ cannot have any cycles, and so $\mu$ cannot have a distinct detour in $F$. It follows that $C^*(F)$ is not $\cZ$-stable by Theorem~\ref{thmA}, case~\ref{A1}. We conclude that $C^*(E)$ is not $\cZ$-stable by Theorem~\ref{dstabperm}\ref{dstabpermstable}, as $C^*(E)\otimes \cK \cong C^*(F)\otimes \cK$ by \cite[Theorem~2.11]{DT05} and \cite[Theorem~2.8]{B77}.
\end{proof}

\section{Graph Algebras with Finitely Many Ideals}\label{finideals}
Outside of the AF case, we may use Theorem~\ref{thmA}, case \ref{A1} to determine when graph algebras with finitely many ideals are $\cZ$-stable. Note that graphs giving rise to such algebras must satisfy Condition~(K), as a subquotient stably isomorphic to $C(T)$ has infinitely many ideals. We first need the following lemma on the permanence of detours.

\begin{lem}\label{detourshereditary}
    Let $E$ be a directed graph, and suppose $H\subset E^0$ is a hereditary, saturated subset. Then $E$ has distinct detours if and only if the graphs $E_{H}$ and $E \setminus H$ have distinct detours.
\end{lem}

\begin{proof}
    We begin with the forward direction and first show that $E_H$ has distinct detours. As $H$ is hereditary, a source in $E_H$ must also be a source in $E$. Let $\mu \in (E_H)^{\leq \infty}$ be a simple path. If $\mu$ is a finite path, then $s(\mu)$ is a source in $E_H$ and therefore a source in $E$. If $E$ is an infinite path in $E_H$, then it is also infinite in $E$. In either case, $\mu$ is a simple path in $E^{\leq\infty}$. By assumption, we may find a distinct detour $\nu \in E^*$ for $\mu$. Because $r(\nu)$ lies on $\mu$ and hence is in $H$, we conclude that $\nu$ lies in $E_H$ since $H$ is hereditary. 
    \par To prove that $E\setminus H$ has distinct detours, consider a simple path $\gamma$ in $(E\setminus H)^{\leq\infty}$. Because $H$ is saturated, every source of $E\setminus H$ is also a source of $E$, so $\gamma$ belongs to $E^{\leq \infty}$ by the same argument as the preceding paragraph. Let $\lambda \in E^*$ be a distinct detour for $\gamma$. If any vertex of $\lambda$ belonged to $H$, then so too would $s(\lambda)$ since $H$ is hereditary. But $s(\lambda)$ is not in $H$ since it lies on $\gamma$, a path in $E\setminus H$. Thus $\lambda$ lies in $E \setminus H$.
    \par For the reverse direction, suppose $E_H$ and $E\setminus H$ have distinct detours. Let $\mu=\mu_1\mu_2\cdots$ (terminating if $\abs{\mu}<\infty$) be a simple path in $E^{\leq\infty}$. If $\mu$ contains a vertex of $H$, say $r(\mu_i)$, then $\mu_i\mu_{i+1}\cdots$ is a path in $(E_H)^{\leq\infty}$ since $H$ is hereditary and sources of $E$ are sources of $E_H$ or $E\setminus H$. Then we may find a distinct detour $\nu \in (E_H)^*$ for $\mu$, which is still a distinct detour for $\mu$ in $E$. On the other hand, if $\mu$ does not contain a vertex of $H$, then $\mu$ is a path in $(E\setminus H)^{\leq\infty}$. Therefore, there is a distinct detour $\lambda \in (E\setminus H)^*$ for $\mu$, which is again a distinct detour for $\mu$ in $E$. In either case, $\mu$ has a distinct detour in $E$, so $E$ has distinct detours.
\end{proof}

With this additional ingredient, we are ready to prove case \ref{A2} of Theorem~\ref{thmA}.
\begin{proof}[Proof of Theorem~\ref{thmA}, case \ref{A2}]
	As was the case in the proof of case \ref{A1}, we have \ref{thmAjiangsu} $\implies$ \ref{thmAsubq} $\implies$ \ref{thmAdetours}. Thus we again only need to show that \ref{thmAdetours} implies \ref{thmAjiangsu}. First, construct a maximal chain of hereditary, saturated subsets $\emptyset = H_0 \subset H_1 \subset \cdots \subset H_n = E^0$, which exists since $C^*(E)$ has finitely many ideals. Then the ideal $I_{H_1}$ is simple and stably isomorphic to $C^*(E_{H_1})$. Similarly, each quotient $C^*(E_{H_i}\setminus H_{i-1})$ is simple. It follows that $C^*(E_{H_1})$ and $C^*(E_{H_i}\setminus H_{i-1})$ are each either AF or purely infinite by \cite[Remark~5.6]{BPRS00}. When they are AF, then they are $\cZ$-stable by Theorem~\ref{thmA}, case~\ref{A1}, and Lemma~\ref{detourshereditary}. On the other hand, when they are purely infinite, then they are nuclear by \cite[Proposition~2.6]{KP99} and hence $\cO_\infty$-stable by \cite[Theorem~3.15]{KP00} (c.f. \cite[Theorem~7.2.6]{RS02}), from which it follows that they are $\cZ$-stable. Since $\cZ$-stability is preserved by stable isomorphism and by taking extensions by parts \ref{dstabpermstable} and \ref{dstabpermideal} of Theorem~\ref{dstabperm}, we may inductively extend along the chain $0 \triangleleft I_{H_1} \triangleleft I_{H_2} \triangleleft \cdots \triangleleft I_{H_n}=C^*(E)$ to conclude that $C^*(E)$ is $\cZ$-stable.
\end{proof}

We will now examine $\cZ$-stability in the case of finite graphs. Recall that by Remark~\ref{detournosources}, any graph with distinct detours is source-free. On the other hand, if $E$ is a finite graph with no sources, then $E^{\leq\infty}$ does not contain any simple paths (since $\abs{E^0}<\infty$), so $E$ trivially has distinct detours. Due to the restrictive nature of such graphs, $\cZ$-stability is a very strong condition on these graph algebras, being equivalent to $\cO_\infty$-stability.

\begin{proof}[Proof of Theorem~\ref{thmB}]
    First, recall that \ref{thmBgraph} implies \ref{thmBpureinf} by \cite[Section~3, Theorem~C]{FC23}. Now, observe that \ref{thmBpureinf} implies \ref{thmBjiangsu} by \cite[Theorem~5]{JS99} and \ref{thmBjiangsu} implies \ref{thmBsubq} by Theorem~\ref{dstabperm}\ref{dstabpermideal}.  It remains to show that \ref{thmBsubq} implies \ref{thmBgraph}. If $v$ has a source, then the ideal $I_v$ generated by $p_v$ is stably isomorphic to $\bC p_v \cong p_vI_vp_v$ by \cite[Theorem~2.8]{B77}, and so $I_v$ is an elementary $C^*$-algebra by Lemma~\ref{stabsubquo}. On the other hand, if $E$ fails to have Condition~(K), then $C^*(E)$ has a subquotient stably isomorphic to $C(T)$, which then yields an elementary $C^*$-algebra as a quotient. Thus, \ref{thmBsubq} implies \ref{thmBgraph}, completing the proof.
\end{proof}

\section{Pure Graph \texorpdfstring{$C^*$}{C*}-Algebras}\label{puregraph}

In proving Theorem~\ref{thmA}, case \ref{A1}, we have completed most of the work necessary to prove Theorem~\ref{thmC}. We therefore immediately proceed to its proof, showing that $E$ having Condition (K) and distinct detours is equivalent to $C^*(E)$ being pure. 
\begin{proof}[Proof of Theorem~\ref{thmC}]
    We begin by proving the equivalence of \ref{thmCgraph} and \ref{thmCsubq}. Suppose $E$ has Condition (K) and distinct detours, and let $B$ be a simple subquotient of $C^*(E)$. If $B$ is not AF, it cannot be elementary, so we may assume that $B$ is an AF algebra. Since $E$ has Condition (K), ideals and quotients are (stably) isomorphic to the graph algebras of particular subgraphs of $E$. Suppose $H\subseteq E^0$ is a hereditary, saturated subset so that $B\cong I_H/J$ for some ideal $J\trianglelefteq I_H$. By Lemma~\ref{detourshereditary}, the graph $E_H$ has distinct detours as well as Condition~(K) by \cite[Section~3, Lemma~A]{FC23}. By \cite[Theorem~5.1]{RT14}, $I_H$ is isomorphic to $C^*(\clo{E}_H)$, where $\clo{E}_H$ is obtained by adding sinks to the graph $E_H$. This operation clearly preserves the presence of distinct detours and Condition~(K). It follows that the quotient $I_H/J$ is isomorphic to $C^*(\clo{E}_H\setminus H')$ for some hereditary saturated subset $H' \subseteq \clo{E}_H^0$. In particular, $\clo{E}_H\setminus H'$ is a graph with distinct detours by Lemma~\ref{detourshereditary}, and $C^*(\clo{E}_H\setminus H')$ is AF since it is isomorphic to $B$. It follows that $B$ is a $\cZ$-stable algebra by Theorem~\ref{thmA}, case \ref{A1}. In particular, $B$ is not an elementary $C^*$-algebra. For the reverse direction, suppose that $C^*(E)$ has no elementary subquotients. By Lemma~\ref{detournec}, it follows that $E$ must have distinct detours. Similarly, if $E$ does not have Condition (K), a standard construction produces a subquotient stably isomorphic to $C(\bT)$. This subsequently yields an elementary subquotient, so $E$ must have Condition (K).
    \par We next show that \ref{thmCgraph} and \ref{thmCsubq} together imply \ref{thmCpure}. Since $C^*(E)$ has no elementary subquotients, it is nowhere scattered by \cite[Theorem~3.1]{TV24}. Since $E$ has Condition (K), $C^*(E)$ has real rank zero by \cite[Theorem~4.1]{JP02}. Thus $C^*(E)$ has the Global Glimm Property by \cite[Proposition~7.4]{TV23}. Combining \cite[Theorem~A]{FC23} and \cite[Theorem~6.5]{APTV24} proves $C^*(E)$ is pure. Conversely, if $C^*(E)$ is pure, $C^*(E)$ has no elementary subquotients by \cite[Proposition~5.2]{APTV24} and \cite[Theorem~7.1]{TV23}.
\end{proof}

This leads us to the following conjecture about $\cZ$-stability for graph algebras.

\begin{conj}\label{zconj}
    Let $E$ be a row-finite graph. Then the following are equivalent:
    \begin{enumerate}
        \item $E$ has Condition (K) and distinct detours.
        \item $C^*(E)$ has no elementary subquotients.
        \item $C^*(E)$ is pure.
        \item $C^*(E)$ is $\cZ$-stable.
    \end{enumerate}
\end{conj}

\begin{rmk}
	If Conjecture~\ref{zconj} is correct, then we can remove the assumption of row-finiteness by adding the condition that infinitely many edges in $r^{-1}(v)$ lie on cycles in $E$ when $v$ is an infinite receiver. The proof is similar to that of Corollary~\ref{zstabAF}, as a Drinen--Tomforde desingularization of $E$ will fail to have distinct detours if an infinite receiver lies on only finitely many cycles.
\end{rmk}


\end{document}